\documentclass[11pt]{article}

\usepackage{amsmath,amsthm}
\usepackage{amsfonts}
\usepackage[nobysame]{amsrefs}
\usepackage{verbatim}
\usepackage{amssymb}
\usepackage{amscd}
\usepackage[colorlinks=true, citecolor=blue, 
]{hyperref}
\usepackage{mathrsfs}

\usepackage{tikz}
\usetikzlibrary{matrix}

\usepackage{fouriernc}

\newtheorem{theorem}{Theorem}
\newtheorem{corollary}[theorem]{Corollary}
\newtheorem{lemma}[theorem]{Lemma}

\newtheorem{example}[theorem]{Example}

\newcommand{\set}[1]{\,\left\{#1\right\}}
\newcommand{\setd}[2]{\,\left\{#1\ \colon\ #2\right\}}
\newcommand{\ndm}[1]{J(#1)}

\newcommand{\RR}{\mathbb{R}}
\newcommand{\ZZ}{\mathbb{Z}}
\newcommand{\CC}{\mathbb{C}}
\newcommand{\NN}{\mathbb{N}}

\newcommand{\opi}{\overline{\pi}}

\newcommand{\wotlim}{\textsc{wot}-\lim}

\usepackage{authblk}
\title{On complements of Kazhdan projections in semisimple groups\thanks{This project has received funding from the European Research Council (ERC) under the European Union's Horizon 2020 research and innovation programme (grant agreement no. 677120-INDEX).}}
\author[ ]{Piotr W. Nowak\thanks{pnowak@impan.pl}}
\author[ ]{Eric Reckwerdt\thanks{ereckwerdt@impan.pl}}
\affil[ ]{Institute of Mathematics\\ Polish Academy of Sciences\\ Warsaw, Poland}
\begin{document}

\maketitle

\begin{abstract}
We prove that for an isometric representation of some groups on certain Banach spaces, 
the complement of the subspace of invariant vectors is 1-complemented. 
\end{abstract}

\section{Introduction}
An isometric representation $\pi$ of a locally compact group $G$ on a reflexive Banach space $E$ induces a direct sum decomposition 
$$E=E^{\pi}\oplus E_{\pi}$$
into the invariant vectors $E^{\pi}$ and its canonical, $\pi$-invariant complement, $E_{\pi}$.  The corresponding projection $P^{\pi}$
from $E$ to $E^{\pi}$ along $E_{\pi}$ has norm 1 under fairly general conditions. 
However, typically,  the complementary projection $I-P^{\pi}$ need not be of  norm 1. 
An easy example is that of the regular representation of a finite group $G$ on $\ell_p(G)$, $p>2$, where $E^{\pi}$ are the constant functions and $E_{\pi}$ are
the functions on $G$ whose mean is zero (see Example \ref{example: projection I-P can have norm arbitrarily close to 2} for details). 

In this article we give conditions under which, for representations $\pi$ on a Banach space $E$ in a certain class $\mathcal{O}$, consisting of uniformly convex
uniformly smooth Banach spaces with the Opial property (see Section \ref{section: duality mappings} for details),
$E_{\pi}$ is 1-complemented and more precisely that
$$\Vert I-P^{\pi}\Vert =1.$$

\begin{theorem}\label{theorem: intro main technical}
Let $G$ be a locally compact group and let $\pi$ be an isometric representation on a Banach space $E\in \mathcal{O}$. Assume that there exists 
a sequence of elements $\set{g_n}\subset G$, $g_n\to \infty$, such that 
$$\wotlim \pi_{g_n}=\wotlim \pi_{g_n^{-1}} =P^{\pi}.$$
Then $E_{\pi}$ is Birkhoff-James orthogonal to $E^{\pi}$ and $\Vert I-P^{\pi}\Vert=1$. 
\end{theorem}

The above theorem in particular applies to class of groups satisfying certain algebraic conditions, introduced as \emph{quasi-semisimple groups} in \cite{bader-gelander}.
They are defined via the existence of a $KAK$-type decomposition and generation by certain contraction subgroups, see \cite{bader-gelander,ciobotaru} for 
details. This class of groups includes the classical semisimple Lie groups as well as certain automorphism groups of trees. 

Recall that a representation $\pi$ of $G$ on a reflexive Banach space $E$  is a $c_0$-representation if the matrix coefficients 
$\langle \pi_g v,w\rangle$ vanish at infinity for
every $v\in E_{\pi}$ and $w\in E^{*}$. As shown in \cite{bader-gelander}, for quasi-semisimple groups this is the case for all isometric representations on reflexive Banach spaces. 
\begin{corollary}\label{corollary: c_0 reps have I-P of norm 1}
Let $\pi$ be a representation of a non-compact locally compact group $G$ on $E\in\mathcal{O}$. If $\pi$ is a $c_0$-representation then 
$\Vert 1-P^{\pi}\Vert =1$.

In particular, if $G$ is quasi-semisimple then the above estimate holds for all isometric representations on $E\in \mathcal{O}$.
\end{corollary}

This is particularly interesting in the context of properties such as property $(T_{\ell_p})$, studied by Bekka and Olivier \cite{bekka-olivier}.
In that case we obtain a stronger conclusion that holds not only in the algebra of bounded operators on $E\in \mathcal{O}$, but in the
the associated maximal group Banach algebra (see Section \ref{section: complements of Kazhdan projections}).

\begin{corollary}\label{corollary: complement of kazhdan projection has norm 1}
Let $G$ be quasi-semisimple. Assume that $G$ satisfies uniform property $(T_{\mathcal{O}})$, namely that there exists 
a Kazhdan projection $p$ in the Banach algebra $C_{\max}^{\mathcal{O}}(G)$. 
Then the element $1-p$ satisfies $\Vert 1-p\Vert=1$. 
\end{corollary}

We remark that in \cite{bernau-lacey} Bernau and Lacey studied projections $P$ such that both $P$ and $I-P$ are of norm 1, and in particular gave
a general description of such projections in $L_p$-spaces.

\subsection*{Acknowledgements} 
We are very grateful to Nicolas Monod and Mikael de la Salle for insightful comments on an earlier version of the preprint and to Uri Bader for illuminating
comments and discussions.

\section{Banach spaces, duality mappings, and the Opial property}

\subsection{Duality mappings}\label{section: duality mappings}
Consider a \emph{duality mapping} $J: E\to \mathcal{S}(E^*)$, into the collection $\mathcal{S}(E^*)$ of subsets of $E^*$, defined by the condition
$$\ndm{v}=\setd{w\in E^*}{\langle v,w\rangle = \Vert v\Vert^2  \text{ and } \Vert w \Vert =\Vert v\Vert }.$$
A thorough discussion of duality mappings in Banach spaces can be found in e.g.  \cite{chidume}.

Denote by $w-\lim$ the limit in the weak topology on $E$. 
A Banach space $E$ has the \emph{Opial property} if for every weakly convergent sequence $\set{ x_n}\subseteq E$ with the weak limit $x_0\in E$ the inequality
$$\liminf \Vert x_n-x_0\Vert <\liminf \Vert x_n-y\Vert,$$
holds for every $y\neq x_0$. Every separable Banach space admits an equivalent norm with the Opial property \cite{vandulst}.

It is known that a uniformly convex uniformly smooth Banach space has Opial's property if and only  if the following condition 
is satisfied: 
for every sequence $\set{x_n}$ in $E$, we have 
$$w-\lim x_n=x\ \ \ \ \ \  \Longleftrightarrow\ \ \ \ \ \   w-\lim J(x_n-x) =0.$$ 
The second
mode of convergence is the so-called $\Delta$-convergence. 
We refer to \cite{solimini-tintarev} for details.

By $\mathcal{O}$ we will denote the class of uniformly convex uniformly smooth Banach spaces with the Opial property. 
It is known for instance that the spaces  $\ell_p$, $1<p<\infty$, belong to the class $\mathcal{O}$, as do their infinite direct  $q$-sums, $1<q<\infty$.
However, the spaces $L_p[0,1]$, $p\in(1,2)\cup (2,\infty)$ do not belong to the class $\mathcal{O}$.

\subsection{Representations} \label{subsection: representations}

Let $B(E)$ denote the space of bounded linear operators on $E$. We will consider $B(E)$ with the strong operator topology (\textsc{sot}) and with the weak operator topology (\textsc{wot}).
Let $G$ be a locally compact group. Let $\pi$ be an isometric representation of $G$ on $E\in \mathcal{O}$ that is continuous in the strong operator topology on $B(E)$. In other words, for every $v\in E$ the orbit map $G\to E$, $g\mapsto \pi_gv$, is continuous into the norm topology on $E$.
The dual space $E^*$ is naturally equipped with an isometric representation $\opi$, defined by the formula 
$$\opi_g=(\pi_g^*)^{-1},$$
for every $g\in G$. If $\pi$ is \textsc{sot}-continuous then so is $\opi$, under the assumption that $E$ is reflexive, see \cite[Proposition 4.1.2.3, page 224]{warner}. 
A \emph{matrix coefficient} of a representation $\pi$ on a Banach space $E$, associated to vectors $v\in E$, $w\in E^*$, is a function $\psi_{v,w}:G\to \CC$ defined by 
$\psi_{v,w}(g)=\langle \pi_gv,w\rangle$. When $E$ is reflexive we will say that $\pi$ is a \emph{$c_0$-representation } if 
$\psi_{v,w}\in C_0(G)$ for every $v\in E_{\pi}$ and $w\in E^*$.

\subsection{Decompositions and projections}
Denote 
$$E^{\pi}=\setd{ v\in E}{\pi_gv=v \text{ for all } g\in G}.$$
If $E$ is reflexive then there is a direct sum decomposition $E=E^{\pi}\oplus E_{\pi}$, where $E_{\pi}$ is the annihilator of the space $(E^*)^{\opi}$.
Since $E_{\pi}$ is $\pi$-invariant, the decomposition is in fact a decomposition of $\pi$ into the trivial representation and its complement.
See e.g. \cite{bfgm,bader-rosendal-sauer}. Note however, that if $E$ is not reflexive then the projection onto $E^{\pi}$ might not be equivariant
\cite[Example 2.29]{bfgm}, and in fact there might not be a bounded projection onto $E^{\pi}$ at all \cite[Theorem 1]{shulman}.

Let $P^{\pi}:E\to E^{\pi}$ be the projection along $E_{\pi}$. The projection $P^{\pi}$ is known to satisfy $\Vert P^{\pi}\Vert=1$ in many
cases, see \cite{paco, shulman}. However, for the complementary projection in general we only have
$$\Vert I-P^{\pi}\Vert \le 2.$$
\begin{example}\label{example: projection I-P can have norm arbitrarily close to 2}\normalfont

 Let $G$ be a finite group. Consider the left regular representation of $G$ on $E= \ell_p(G)$, $1\le p\le \infty$. 
The invariant vectors are then the constant functions on $G$ and the complement $E_{\pi}$ is the subspace of functions satisfying 
$$Mf= \dfrac{1}{\#G}\sum_{h\in G} f(h)=0.$$
In the case $p=\infty$ and the group $\ZZ_n$, $n\ge 3$ it is easy to see that  $\Vert I-P^{\pi}\Vert> 1$ by considering the vector
$f=(1,1,\dots,1,-1)\in \ell_{\infty}(\ZZ_n)$. Then $\Vert f\Vert =1$, $Mf=1-2/n$ and $\Vert f-Mf\Vert=2-2/n$.
Since $G$ is finite, clearly, by choosing $n\in \NN$ and $2<p<\infty$, sufficiently large we obtain that the norm of $\Vert I-P^{\pi}\Vert$
can be arbitrarily close to 2. The same is true for any group with a finite quotient of cardinality at least 3.

Observe that if $G$ is residually finite with a sequence $\set{N_i}$ of finite index normal subgroups, $\bigcap N_i=\set{e}$, then 
the previous case shows that for the representation of $G$ on $\ell_p(\coprod G/N_i)=(\bigoplus\ell_p(G/N_i))_{(p)}$ for $p=\infty$ the norm $\Vert I- P^{\pi}\Vert$ 
will be in fact 2.
We can thus choose $p$ such that the
norm of the projection $I-P^{\pi}$ will be arbitrarily close to 2.\qed
\end{example}

\section{Proofs}

Before proving the main theorem we first need a few lemmas regarding duality mappings and their behavior with respect to isometric representations within the class $\mathcal{O}$ of Banach spaces.
For an invertible isometry $S$ on $E$ denote $\overline{S}=(
S^*)^{-1}$.
\begin{lemma}\label{lemma: duality map commutes with isometries}
Let $E\in \mathcal{O}$ be a Banach space and let $S\in B(E)$ be an invertible isometry. Then 
$\overline{S}\ndm{v}=\ndm{Sv}$. 
\end{lemma}
\begin{proof}
    First note that since $S$ is an invertible isometry, so is $S^*$, and we have that 
     $\Vert S^* \ndm{Sv}\Vert = \Vert  \ndm{Sv}\Vert=\Vert Sv\Vert=\Vert v\Vert$. 
Thus 
\begin{align*}
    \langle v,S^* \ndm{Sv}\rangle = \langle S v,\ndm{S v}\rangle & = \Vert S v \Vert^2 = \Vert v\Vert^2=\langle v,\ndm{v}\rangle.
\end{align*}
By uniform convexity and uniform smoothness of $E$ the duality mapping $J$ is bijective, which 
guarantees that there is a unique element $w\in E^*$ such that 
$\Vert w\Vert=\Vert v\Vert$ and $\langle v,w\rangle=\Vert v\Vert^2$. This yields
$$S^* \ndm{S v}=\ndm{v},$$
and the claim is proved.\end{proof}

A vector $0\neq v\in E$ is \emph{Birkhoff-James orthogonal} to $0\neq w\in E$ if 
$$\Vert v\Vert \le \Vert v+\lambda w\Vert$$
for every $\lambda\in \RR$. Given subspaces $V_1,V_2$ in $E$ we will say that
$V_1$ is Birkhoff-James orthogonal to $V_2$ if for every $v_1\in V_1$ and $v_2\in V_2$, $v_1$ is Birkhoff-James orthogonal to $v_2$.

A useful lemma  by Kato \cite{kato} states that $\Vert v\Vert \le \Vert v+\lambda w\Vert$ for every $\lambda >0$ if and only if we have
\begin{equation}\label{equation: Kato's condition}
\langle w, \ndm{v}\rangle \ge 0.
\end{equation}
In particular this implies that $v$ is Birkhoff-James orthogonal to $w$ if we have 
$\langle w,\ndm{v}\rangle\ge 0$ and $\langle w,\ndm{-v}\rangle\ge 0$. 
Of course, this is equivalent to $\langle w,\ndm{v}\rangle =0$.
{
    \renewcommand{\thetheorem}{\ref{theorem: intro main technical}}

\begin{theorem}
Let $G$ be a locally compact group and let $\pi$ be an isometric representation on a Banach space $E\in \mathcal{O}$. Assume that there exists 
a sequence of elements $\set{g_n}\subset G$, $g_n\to \infty$, such that 
$$\wotlim \pi_{g_n} = \wotlim \pi_{g_n}^{-1} =P^{\pi}.$$
Then $E_{\pi}$ is Birkhoff-James orthogonal to $E^{\pi}$ and $\Vert I-P^{\pi}\Vert=1$. 
\end{theorem}

}

\begin{proof}
Let $w\in E^{\pi}$, $v\in E_{\pi}$ be arbitrary. We first observe that since $P^{\pi} =\wotlim \pi_{g_n}$, then 
$$(P^{\pi})^*=\wotlim \pi_{g_n}^*=\wotlim \opi_{g_n}^{-1}.$$
Thus we have 
\begin{align*}
\langle w,\ndm{v}\rangle &=\langle P^{\pi} w,\ndm{v}\rangle\\
&=\langle w,(P^{\pi})^*\ndm{v}\rangle\\
&= \lim \left\langle w, \opi_{g_n^{-1}} \ndm{v}\right \rangle\\
&= \lim \left\langle w, \ndm{ \pi_{g_n^{-1}} v} \right\rangle,
\end{align*}
where the last equality follows from Lemma \ref{lemma: duality map commutes with isometries}.

Now, by assumption, the sequence $\pi_{g_n}^{-1} v$ converges weakly to $P^{\pi}v=0$ since $v\in E_{\pi}$.
By the Opial property we have that this is the same as the condition that 
$$\ndm{\pi_{g_n}^{-1}v }\to 0$$
in the weak topology in $E^*$.
In particular, 
$$\lim \left\langle w, \ndm{\pi_{g_n}v}\right\rangle =\langle w,\ndm{v}\rangle=0.$$
The norm estimate $\Vert I-P^{\pi}\Vert=1$ is then a direct consequence of Birkhoff-James orthogonality.
\end{proof}

The above theorem applies to the class of \emph{quasi-semisimple groups}, studied in \cite{bader-gelander}. 
These groups are defined as satisfying two algebraic conditions: a version of a $KAK$-decomposition and that the group is generated by
certain contraction subgroups, associated to sequences of elements tending off to infinity. 
We refer to \cite{bader-gelander} and \cite{ciobotaru} for definitions and applications.

\begin{proof}[Proof of Corollary \ref{corollary: c_0 reps have I-P of norm 1}]
If $\pi$ is an isometric $c_0$-representation of $G$ on a reflexive Banach space $E$ then, for $v\in E_{\pi}$ and $w\in E^*$, every matrix coefficient of $\pi$ of the form $\psi_{v,w}=\langle \pi_g v, w\rangle$
 vanishes as $g$ goes to infinity. 
This in particular means, that the \textsc{wot} closure of $\pi(G)$ in $B(E)$ is the one-point compactification of $G$, with the point at infinity 
being precisely the projection $P^{\pi}$. 
In particular, this ensures that the assumptions of Theorem \ref{theorem: intro main technical} are satisfied.

The class of quasi-semisimple groups includes the classical semisimple Lie groups and certain automorphism groups of trees. 
As proved in \cite{bader-gelander}, quasi-semisimple groups satisfy a Veech decomposition of their space $W(G)$ of weakly almost periodic functions,
$$W(G)=C_0(G)\oplus \CC.$$
(See also \cite{veech} for the classical result of Veech on semisimple Lie groups.)
Every matrix coefficient of an isometric representation on a reflexive Banach space $E$ is an element of $W(G)$, 
with the projection (invariant mean on weakly almost periodic functions) $m:W(G)\to \CC$ given by 
$m(\psi_{v,w})=\langle P^{\pi}v,P^{\opi}w\rangle$. 
Thus, the previous argument applies to every isometric representation $\pi$ of a quasi-semisimple group on $E\in\mathcal{O}$.
\end{proof}

\subsection{Complements of Kazhdan projections}\label{section: complements of Kazhdan projections}

Let $\mathcal{O}$ be a class of Banach spaces. The group ring $\CC G$ can be completed in the the norm
$$\Vert f\Vert =\sup \Vert \pi(f)\Vert,$$
where the supremum is taken over all isometric representations of $G$ on Banach spaces $E\in \mathcal{O}$. 
This completion is a Banach algebra, which we denote $C^{\mathcal{O}}_{\max}(G)$. In the case when $\mathcal{O}$ consists of the class of Hilbert spaces the algebra is the 
maximal group $C^*$-algebra of $G$. 
A Kazhdan projection $p\in C^{\mathcal{O}}_{\max}(G)$ is an idempotent such that $\pi(p)=P^{\pi}$ for every isometric representation of $G$ on $E\in \mathcal{O}$.
$G$ has uniform property $(T_{\mathcal{O}})$ if a Kazhdan projection exists in $p\in C^{\mathcal{O}}_{\max}(G)$.
See \cite{drutu-nowak}
for a detailed study.

\begin{proof}[Proof of Corollary \ref{corollary: complement of kazhdan projection has norm 1}]
Follows from the fact that for every isometric representation of $G$ on $E\in \mathcal{O}$ we have 
$\Vert I-P^{\pi}\Vert= 1$.
\end{proof}


\begin{bibdiv}
\begin{biblist}



\bib{bfgm}{article}{
   author={Bader, U.},
   author={Furman, A.},
   author={Gelander, T.},
   author={Monod, N.},
   title={Property (T) and rigidity for actions on Banach spaces},
   journal={Acta Math.},
   volume={198},
   date={2007},
   number={1},
   pages={57--105},
}

\bib{bader-gelander}{article}{
   author={Bader, U.},
   author={Gelander, T.},
   title={Equicontinuous actions of semisimple groups},
   journal={Groups Geom. Dyn.},
   volume={11},
   date={2017},
   number={3},
   pages={1003--1039},
}


\bib{bader-rosendal-sauer}{article}{
   author={Bader, U.},
   author={Rosendal, C.},
   author={Sauer, R.},
   title={On the cohomology of weakly almost periodic group representations},
   journal={J. Topol. Anal.},
   volume={6},
   date={2014},
   number={2},
   pages={153--165},
}

\bib{bekka-olivier}{article}{
   author={Bekka, B.},
   author={Olivier, B.},
   title={On groups with property $(T_{\ell_p})$},
   journal={J. Funct. Anal.},
   volume={267},
   date={2014},
   number={3},
   pages={643--659},
}

\bib{bernau-lacey}{article}{
   author={Bernau, S. J.},
   author={Lacey, H. Elton},
   title={Bicontractive projections and reordering of $L_{p}$-spaces},
   journal={Pacific J. Math.},
   volume={69},
   date={1977},
   number={2},
   pages={291--302},
}


\bib{browder}{article}{
   author={Browder, F. E.},
   title={Fixed point theorems for nonlinear semicontractive mappings in
   Banach spaces},
   journal={Arch. Rational Mech. Anal.},
   volume={21},
   date={1966},
   pages={259--269},
}


\bib{brown-guentner}{article}{
   author={Brown, N.},
   author={Guentner, E.},
   title={Uniform embeddings of bounded geometry spaces into reflexive
   Banach space},
   journal={Proc. Amer. Math. Soc.},
   volume={133},
   date={2005},
   number={7},
   pages={2045--2050},
}		
	
\bib{chidume}{book}{
   author={Chidume, C.},
   title={Geometric properties of Banach spaces and nonlinear iterations},
   series={Lecture Notes in Mathematics},
   volume={1965},
   publisher={Springer-Verlag London, Ltd., London},
   date={2009},
   pages={xviii+326},
}

\bib{ciobotaru}{article}{
   author={Ciobotaru, C.},
   title={A unified proof of the Howe-Moore property},
   journal={J. Lie Theory},
   volume={25},
   date={2015},
   number={1},
   pages={65--89},
}
		
	

\bib{drutu-nowak}{article}{
   author={Dru\c{t}u, C.},
   author={Nowak, P. W.},
   title={Kazhdan projections, random walks and ergodic theorems},
   journal={Crelle's Journal, to appear},
 
}


\bib{vandulst}{article}{
   author={van Dulst, D.},
   title={Equivalent norms and the fixed point property for nonexpansive
   mappings},
   journal={J. London Math. Soc. (2)},
   volume={25},
   date={1982},
   number={1},
   pages={139--144},
}
	

\bib{paco}{article}{
   author={Garc\'\i a-Pacheco, F. J.},
   title={Complementation of the subspace of $G$-invariant vectors},
   journal={J. Algebra Appl.},
   volume={16},
   date={2017},
   number={7},
   pages={1750124, 7},
}


	
\bib{kato}{article}{
   author={Kato, T.},
   title={Nonlinear semigroups and evolution equations},
   journal={J. Math. Soc. Japan},
   volume={19},
   date={1967},
   pages={508--520},
}


\bib{lamperti}{article}{
   author={Lamperti, J.},
   title={On the isometries of certain function-spaces},
   journal={Pacific J. Math.},
   volume={8},
   date={1958},
   pages={459--466},
} 


\bib{lindenstrauss-tzafriri-1}{book}{
   author={Lindenstrauss, J.},
   author={Tzafriri, L.},
   title={Classical Banach spaces. I},
   note={Sequence spaces;
   Ergebnisse der Mathematik und ihrer Grenzgebiete, Vol. 92},
   publisher={Springer-Verlag, Berlin-New York},
   date={1977},
   pages={xiii+188},
   }
   
   
   
\bib{opial}{article}{
   author={Opial, Zdzis\l aw},
   title={Weak convergence of the sequence of successive approximations for
   nonexpansive mappings},
   journal={Bull. Amer. Math. Soc.},
   volume={73},
   date={1967},
   pages={591--597},
   issn={0002-9904},
   review={\MR{0211301}},
}
   
 \bib{shulman}{article}{
   author={Shulman, T.},
   title={On subspaces of invariant vectors},
   journal={Studia Math.},
   volume={236},
   date={2017},
   number={1},
   pages={1--11},
}
		
\bib{solimini-tintarev}{article}{
   author={Solimini, S.},
   author={Tintarev, C.},
   title={Concentration analysis in Banach spaces},
   journal={Commun. Contemp. Math.},
   volume={18},
   date={2016},
   number={3},
   pages={1550038, 33},
}		

\bib{veech}{article}{
   author={Veech, W. A.},
   title={Weakly almost periodic functions on semisimple Lie groups},
   journal={Monatsh. Math.},
   volume={88},
   date={1979},
   number={1},
   pages={55--68},
}


\bib{warner}{book}{
   author={Warner, G.},
   title={Harmonic analysis on semi-simple Lie groups. I},
   note={Die Grundlehren der mathematischen Wissenschaften, Band 188},
   publisher={Springer-Verlag, New York-Heidelberg},
   date={1972},
   pages={xvi+529},
}
	

\end{biblist}
\end{bibdiv}
\end{document}